\documentclass [11pt]{amsart}
\usepackage {amsmath, amssymb, amscd, mathrsfs, url, pinlabel, hyperref, verbatim}
\usepackage[text={5.5in,9in},centering,letterpaper,dvips]{geometry}
\usepackage{color,multirow,dcpic,latexsym,pictexwd,graphicx,epstopdf,mathtools}
\usepackage{verbatim}
\usepackage{tikz}
\usepackage{blindtext}
\usetikzlibrary{positioning, arrows.meta}
\newcommand{\here}[2]{\tikz[remember picture]{\node[inner sep=0](#2){#1}}}

\newtheorem {theorem}{Theorem}
\newtheorem {lemma}[theorem]{Lemma}
\newtheorem {proposition}[theorem]{Proposition}
\newtheorem {corollary}[theorem]{Corollary}

\theoremstyle{remark}

\newtheorem {remark}[theorem]{Remark}

\numberwithin{equation}{section}
\numberwithin{theorem}{section}

\title{Ozsv\'ath-Szab\'o $d$-invariants of Almost Simple Linear Graphs}
\author{\c{C}a\u gr{\i} Karakurt }
\address{Department of Mathematics, Bo\u{g}azi\c{c}i University, Bebek 34342}
\email{\href{mailto:cagri.karakurt@boun.edu.tr}{cagri.karakurt@boun.edu.tr}}
\author{O{\u{g}}uz \c{S}avk}
\address{Department of Mathematics, Bo\u{g}azi\c{c}i University, Bebek 34342}
\email{\href{mailto:oguz.savk@boun.edu.tr}{oguz.savk@boun.edu.tr}}
\date{}

\begin{document}
 
\begin{abstract}
We describe an effective method for simultaneously computing $d$-invariants of infinite families of  Brieskorn spheres $\Sigma(p,q,r)$ with $pq+pr-qr=1$.
\end{abstract}
\maketitle

\section{Introduction}
In \cite{OS03b}, Ozsv\'ath and Szab\'o introduced the $d$-invariant, which is a numerical invariant of $\mathrm{spin}^c$ three-manifolds obstructing the existence of rational homology cobordisms. Since then, this invariant has been successfully used for answering a number of questions in three-dimensional topology and knot theory, see for example \cite{OS06}, \cite{MO07}, and \cite{S12}.

Throughout $p$, $q$ and $r$ denote pairwise relatively prime, ordered, positive integers satisfying 
\begin{equation}
\label{pqr}
pq+pr-qr=1.
\end{equation}
The purpose of the present paper is to compute the Ozsv\'ath-Szab\'o $d$-invariant of the Brieskorn sphere $Y=\Sigma(p,q,r)$ which is the link of singularity $x^p + y^q + z^r = 0$. This manifold is the plumbed three-manifold corresponding  to the graph shown in Figure~\ref{fig:plumb}. We shall call such graphs \emph{almost simple linear graphs} (or ASL-graphs for short). The most well-known example of an ASL-graph is the $E_8$ graph which gives the plumbing description of the Poincar\'e homology sphere $\Sigma(2,3,5)$.

\begin{figure}[ht]  \begin{center}
		\includegraphics[width=0.5\textwidth]{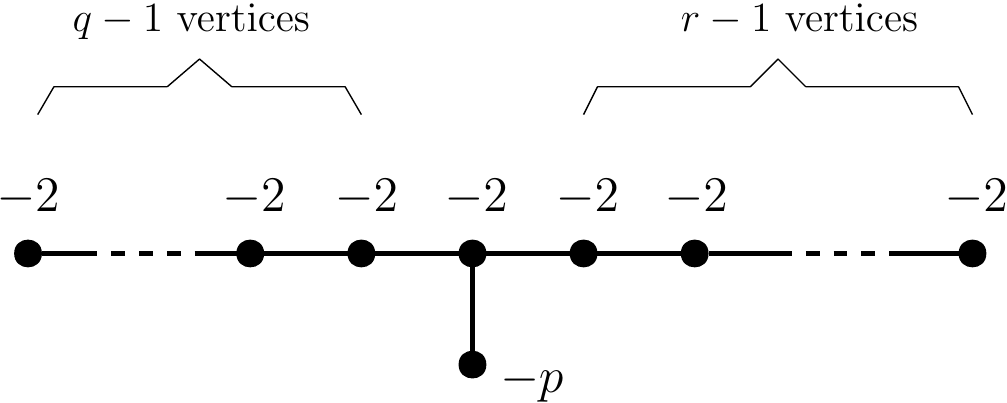}       
		\caption{The plumbing graph for $Y$}
\label{fig:plumb}
	\end{center}
\end{figure}

Combinatorial techniques for computing $d$-invariants of larger classes of plumbed three-manifolds were previously known by Ozsv\'ath and Szab\'o \cite{OS03a} and N\'emethi \cite{Nem05}. With the help of computers, these techniques successfully determine $d$-invariant of a single plumbed three-manifold. Except a few notable works \cite{Nem07}, \cite{BN13}, \cite{E13}, and \cite{LE18}, they usually yield intractable result when one tries to find $d$-invariants of infinite families simultaneously. Our approach is based on the Ozsv\'ath-Szab\'o method but we simplified it greatly for the case of ASL-graphs. Basically, Ozsv\'ath-Szab\'o method requires finding the maximum value of a quadratic form on a lattice whose dimension equals the number of vertices in the graph. We show that for ASL-graphs the maximum is achieved in a bounded region contained  in a two dimensional lattice.  Before  stating our main result, first we note that when $p$ is even the $d$-invariant of $\Sigma (p,q,r)$ can easily be shown to be equal to $(q+r)/4$ (see Proposition~\ref{deven}), so the interesting examples arise only when $p$ is odd. Now define a function \begin{align}
\label{eq:fdefn} 
f(x,y)=-(q+r)x^2+4qxy-4(q-p)y^2-4y,
\end{align}
\noindent Consider the region $$R=\{(x,y)\,:\, -p\leq x\leq p,\; 0\leq y\leq (p-1)/2,\; \text{and } f(1,1)\leq f(x,y)\}\subset \mathbb{R}^2.$$ Also call elements of $\mathfrak{L}=\{-p,-p+2,\dots,p\}\times \{0,1,\dots,(p-1)/2\}$ \emph{lattice points}. For the convenience of the reader, we draw the region $R$ together with the lattice points in Figure~\ref{fig:region} for the triplet $(p,q,r) = (2n+1,2n+2,4n^2+6n+1)$ and $n=5$.

\begin{figure}[ht]  \begin{center}
		\includegraphics[width=0.8\textwidth]{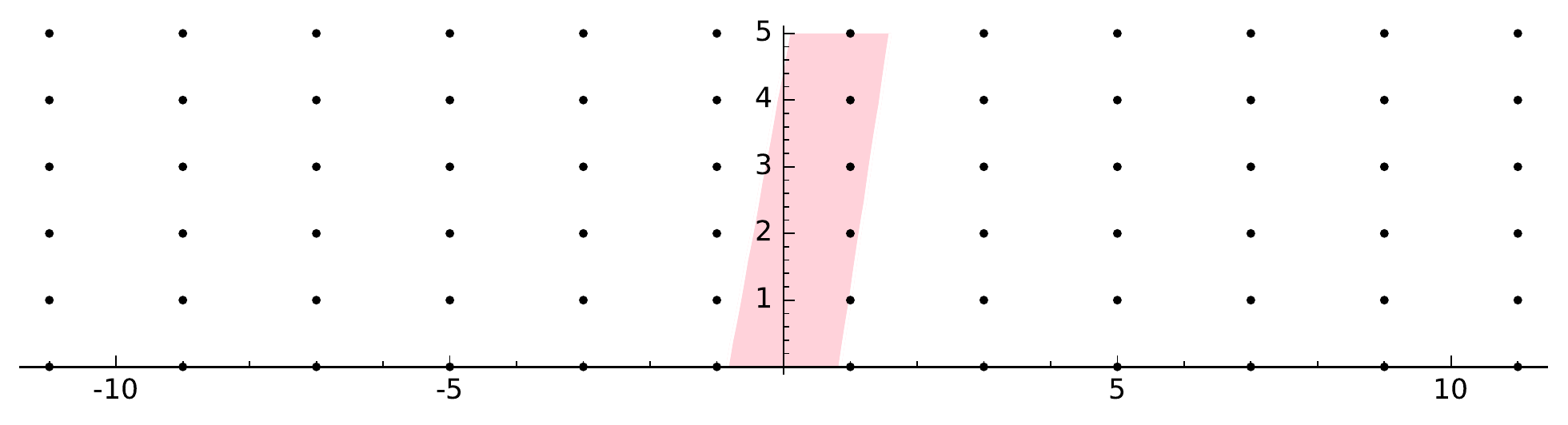}       
		\caption{The region $R$ for $(p,q,r)=(11,12,131)$}
\label{fig:region}
	\end{center}
\end{figure}

\begin{theorem} 
\label{theo:latt} 
For an odd $p$, the $d$-invariant of the Brieskorn sphere $\Sigma (p,q,r)$ is $$d(\Sigma (p,q,r))= \frac{1}{4}\left [ \left (\underset{(a,m)\in \mathfrak{L}}{\mathrm{max}}f(a,m)\right )+q+r\right ] = \frac{1}{4}\left [ \left (\underset{(a,m)\in \mathfrak{L}\cap R}{\mathrm{max}}f(a,m)\right )+q+r\right ] .$$
\end{theorem}

Note that the lattice point $(1,1)$ is always contained in the region $R$, so the maximum is taken on a non-empty set in the latter equation. Describing this maximum for general ASL-graphs is still highly non-trivial but by plugging in $(a,m)=(1,1)$ we immediately see $d(\Sigma(p,q,r)) \geq p-1 >0$. Since $d$-invariant is a rational homology cobordism invariant and it provides an invariant of rational concordance classes of knots by passing to double branched covers, we obtain the following obstruction.

\begin{corollary} 
\label{theo:rball} 
The Brieskorn sphere $\Sigma (p,q,r)$ does not bound a rational homology ball. Consequently, the pretzel knot $K(-p,q,r)$ is not rationally slice.
\end{corollary}
 
The Brieskorn sphere $\Sigma(p,q,r)$ has a non-trivial Fintushel-Stern invariant, so it was previously known that $\Sigma(p,q,r)$ does not bound any $\mathbb{Z}/2\mathbb{Z}$-acyclic four-manifold and hence  the pretzel knot $K(-p,q,r)$ is not smoothly slice \cite{FS85}. A recent work of Issa and McCoy \cite{IM18} shows non-trivial Fintushel-Stern invariant implies $d$-invariant is non-zero. So the above corollary can be inferred from their work as well.

By specializing to certain subfamilies and analysing their lattice points we are able to obtain the exact values of $d$-invariants. To illustrate this point, we give explicit examples of such families. 

\begin{theorem} \label{theo:d-inv}
For $n \geq 1$ we have
\begin{enumerate}
\item $d(\Sigma(2n+1,4n+1,4n+3)) = 2n$, \label{1}
\item $d(\Sigma(2n+1,3n+2,6n+1)) = 2n$,
\item $d(\Sigma(2n+1,3n+1,6n+5)) = 2n$,
\item $d(\Sigma(4n+3,5n+4,20n+11) = 6n+2$,
\item $d(\Sigma(2n+1,2n+2,4n^2+6n+1)) = n^2 +n.$
\end{enumerate}
\end{theorem}

The item \eqref{1} of our theorem was first claimed without proof in a joint work of the first author with Hom and Lidman as a side remark \cite[Remark 3.3]{HKL16}, note that our orientation convention is the opposite of \cite{HKL16}. Subsequently, the remark was used directly or indirectly by a considerable number of authors \cite[Theorem~1.6]{St17}, \cite[Theorem~1.7]{DM19}, \cite[Theorem~1.1]{DS17}, \cite[Theorem~1.7]{AKS19} and \cite[Theorem~1.4]{DM19} in showing or using that the $\Sigma(2n+1,4n+1,4n+3)$ form a linearly independent family in the integral homology cobordism group. Further in \cite[Theorem~1.1]{DHST18}, the $\Sigma(2n+1,4n+1,4n+3)$ was shown to be first family of integral homology spheres generating an infinite-rank summand in the integral homology cobordism group. None of the aforementioned results relied specifically on $\Sigma(2n+1,4n+1,4n+3)$ but they require a special part of Heegaard Floer homology of the members to be sufficiently complicated.
     
In a forthcoming work \cite{KS19}, computing the connected Heegaard Floer homology of such Brieskorn spheres and analysing their almost local equivalence classes in the sense of \cite{HHL18} and \cite{DHST18} respectively, we will prove that they also generate an infinite-rank summand in the integral homology cobordism group. Using the corresponding pretzel knots, we will show related results in the smooth concordance group of knots.

Thanks to recent work of the first author with Lidman and Tweedy \cite{KLT19}, the effect of splicing two Seifert fibers on $d$-invariance is well understood. In particular, combining Theorem~\ref{theo:d-inv} with techniques of \cite[Section 4]{KLT19} one can compute $d$-invariants of large classes of graph homology spheres.
 
\subsection*{Organization} The structure of the paper is as follows. In Section~\ref{Plumbings} and Section~\ref{HFHofPlumb}, we present preliminary notions about plumbings and their Ozsv\'ath-Szab\'o $d$-invariant. In Section~\ref{simplelineargraphs}, we introduce simple and almost simple linear graphs, and analyse their good full paths. Then we compute Ozsv\'ath-Szab\'o $d$-invariant of almost simple linear graphs. By investigating their lattice points, we finally carry out some explicit calculations in Section~\ref{calculation}.

\section{Preliminaries on Plumbings}
\label{Plumbings}
A \emph{plumbing graph} $G$ is a weighted tree with vertex set $\mathrm{Vert}(G)$ such that each vertex $v_i$ is decorated by an integer $e_i$ called the \emph{weight} of $v_i$. A plumbing graph gives rise to a simply connected four-manifold with boundary $X(G)$ which is obtained by plumbing together a collection of disk bundles over the two-sphere, so that the Euler number of the disk bundle corresponding to the vertex $v_i$ is given by weight $e_i$. Let $Y=Y(G)$ be the three-manifold which is the boundary of $X=X(G)$. Then $Y$ is said to be a \emph{plumbed three-manifold}.

The second homology group $H_2(X; \mathbb{Z})$ is the lattice $\mathcal{L}$ generated by $\mathrm{Vert}(G)$, and a symmetric bilinear form of $\mathcal{L}$ is given by the associated \emph{intersection matrix} $I=(a_{ij})$ with entries as follows
\[ a_{ij} = \begin{cases} 
      e_i, & \text{if} \ v_i=v_j, \\
      1, & \text{if} \ v_i \ \text{and} \ v_j \ \text{is connected by one edge}, \\
      0, & \text{otherwise}. 
   \end{cases}
\]
We call $G$ a \emph{negative definite graph} if the corresponding intersection matrix is negative-definite, i.e., $\mathrm{signature}(I)=-|G|$, where $|G|$ denotes the number of vertices of $G$. If $\mathrm{det}(I)= \pm 1$, then $G$ is said to be \emph{unimodular}. In this paper, all our graphs are negative-definite and unimodular except in Section~\ref{simpleg} where we study the plumbings corresponding to lens spaces.

Let $p,q$ and $r$ be pairwise relatively prime, ordered, positive integers. A \emph{Brieskorn integral homology sphere} $\Sigma(p,q,r)$ is the boundary of the unimodular negative-definite star-shaped graph with three branches as shown in Figure~\ref{fig:star}. Note that the plumbing graph is a good resolution graph of the singularity $x^p + y^q + z^r = 0$ whose link is Brieskorn sphere $\Sigma(p,q,r)$.

By \cite[Section 1.1]{Sav02}, the integer weights in the graph are found as follows. One first investigates unique integers $e_0,p',q',r'$ solving the Diophantine equation
\begin{equation}
e_0pqr+p'qr+pq'r+pqr'=-1
\end{equation} 
where $1\leq p' \leq p-1$, $1\leq q' \leq q-1$ and $1\leq r' \leq r-1$.  Then the weights $t_{ij}$ are found from the ratios
\begin{align*}
p/p' = [t_{11}, t_{12}, \ldots,t_{1m_{1}}] \\
q/q' = [t_{21}, t_{22}, \ldots,t_{2m_{1}}] \\
r/r' = [t_{31}, t_{32}, \ldots,t_{3m_{3}}]
\end{align*}
where $[t_{i1},\ldots,t_{im_{i}}]$ denotes the continued fraction
\begin{equation*}
[t_{i1},t_{i2}, \ldots,t_{im_{i}}]=t_{i1}-\cfrac{1}{t_{i2}-\cfrac{1}{\cdots-\cfrac{1}{t_{im_{i}}}}}
\end{equation*}

\begin{figure}[ht]  \begin{center}
		\includegraphics[width=0.3\textwidth]{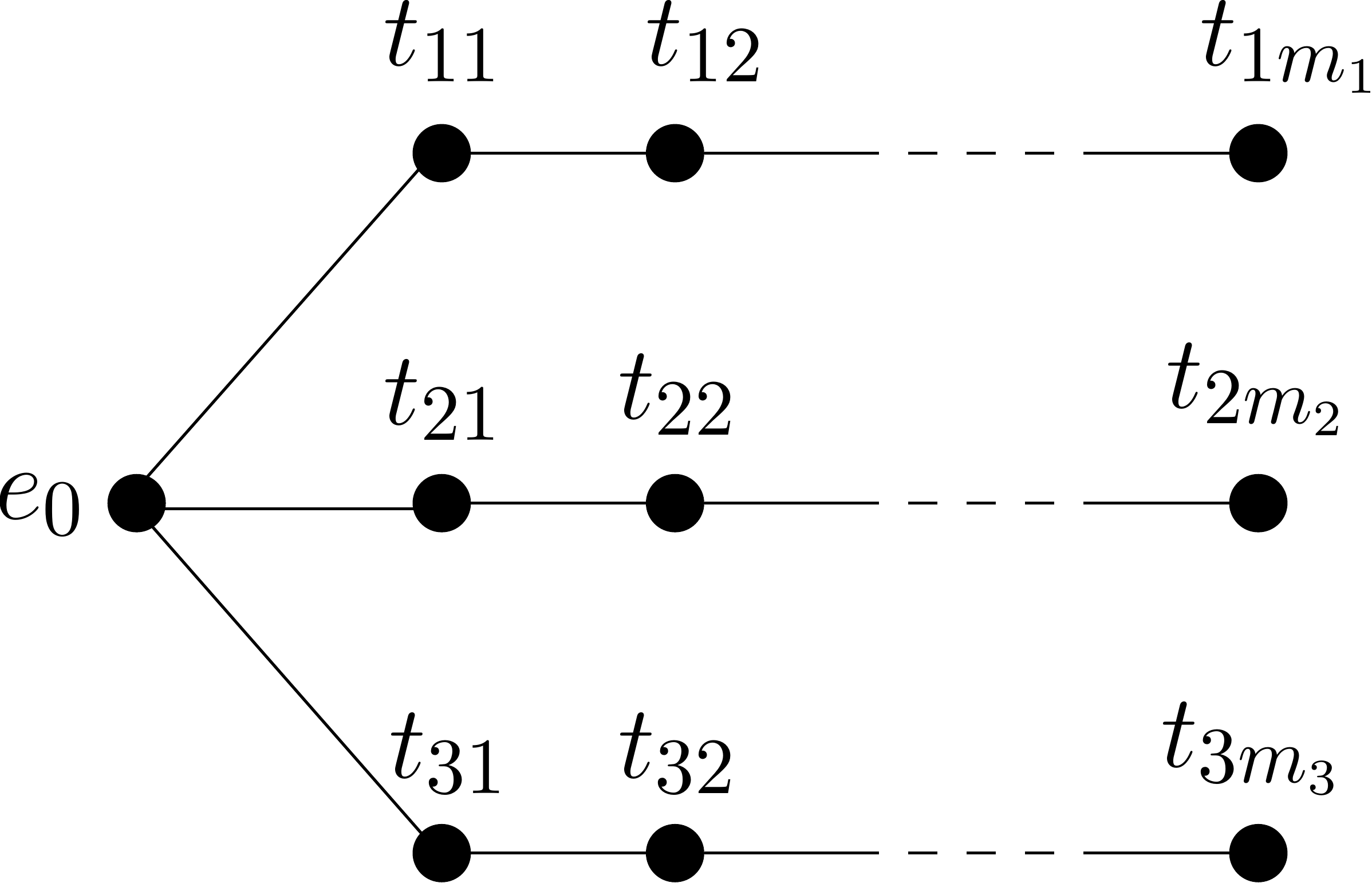}       
		\caption{The star-shaped graph}
\label{fig:star}
	\end{center}
\end{figure}

For the lattice $\mathcal{L}$, we have the short exact sequence
\begin{equation}
0 \longrightarrow \mathcal{L} \xrightarrow{\mathrm{PD}} \mathcal{L}' \longrightarrow H \longrightarrow 0  
\end{equation}
where $\mathcal{L}'$ is the dual lattice $\mathrm{Hom}_{\mathbb{Z}}(\mathcal{L},\mathbb{Z})\cong H^2(X;\mathbb{Z})\cong H_2(X,Y;\mathbb{Z})$, and $H=H_1(Y;\mathbb{Z})$, the latter is trivial if $G$ is unimodular.  Here, $\mathrm{PD}(x)$ is the unique element in $\mathcal{L}'$ which evaluates on each $y \in \mathcal{L}$ as $\langle \mathrm{PD}(x), y \rangle =x^tIy$ by treating $x$ and $y$ as column matrices and $x^t$ denotes the transpose of $x$. 

We say that $k\in  \mathcal{L}'$ is $\textit{characteristic cohomology class}$ if for every vertex $v_j$, we have $\langle k, v_j \rangle +e_j = 0 \text{ mod } 2$. The set of characteristic cohomology classes is denoted by $\mathrm{Char}(G)$.

Let $k^2$ denote the square of the characteristic cohomology class $k$, $I^{-1}$ the inverse of the intersection matrix $I$. It is calculated by the formula $k^2=k^tI^{-1}k$ by treating $k$ as a column matrix and $k^t$ as its transpose. This quantity is called \emph{degree} and it is denoted by $\mathrm{deg}(k)$. Note that the degree of characteristic cohomology classes in the same full path are equal. Further, $\mathrm{deg}(k) = \mathrm{deg}(-k)$.

To calculate $k^2$, one must find the inverse of the intersection matrix $I$. Due to the work of N\'emethi and  Nicolaescu in \cite[Section 5]{NN05}, there is a simple way for finding entries of $I^{-1}$.

For any $v,w \in \mathrm{Vert}(G)$, let $I^{-1}_{vw}$ denotes the $(v,w)$-entry of the intersection matrix $I^{-1}$. Since $I$ is negative definite and $G$ is connected, $I^{-1}_{vw} <0$ for each entry $v,w$. Also as $I$ is described by plumbing graph with integer weights, we can interpret these entries in the following way. For any two vertex $v,w \in \mathrm{Vert}(G)$, let $p_{vw}$ be the unique minimal path in $G$ connecting $v$ and $w$. Let $I_{(vw)}$ be the matrix obtained from $I$ deleting all the rows and columns corresponding to the vertices on the path $p_{vw}$, i.e, $I_{(vw)}$ is the intersection matrix of the complement graph of the path $p_{vw}$. Then $I^{-1}_{vw} = -|\mathrm{det}(I_{(vw)})/ \mathrm{det}(I)|$.

The determinants can be easily computed by using the recipe in \cite[Chapter 5.21]{EN85}. For example,
\begin{itemize}
\item[(1)] The determinant of the negative-definite linear graph associated to $\frac{p}{p'} = [t_{11}, t_{12}, \ldots,t_{1n_{1}}]$ is $(-1)^n p$,
\item[(2)] The determinant of the disconnected graph is the product of determinants of its connected components, 
\item[(3)] If $G$ is a star-shaped graph with central weight $e_0$ and three arms $\frac{p}{p'}, \frac{q}{	q'}, \frac{r}{r'}$ then the absolute value of its determinant is the numerator of $$e_0-\frac{p'}{p}-\frac{q'}{q}-\frac{r'}{r}.$$
\end{itemize}

\section{Preliminaries on $d$-invariant of plumbed three-manifolds}
\label{HFHofPlumb}
In 2003, Ozsv\'ath and Szab\'o compute Heegaard Floer homology of plumbed three-manifolds in \cite{OS03a}. In particular, they describe a purely combinatorial algorithm to partially compute the plus flavor of Heegaard Floer homology $HF^+(-Y)$, where minus sign denotes the reverse orientation. They also give a practical formula for the $d$-invariant. Here, we review the part of Ozsv\'ath and Szab\'o's work related to the $d$-invariant computation. 

\subsection{Ozsv\'ath-Szab\'o algorithm}
Start with an \emph{initial} characteristic cohomology class $k$ satisfying
\begin{equation}
\label{adjunction}
e_j+2 \leq \langle k,v_j \rangle \leq -e_j
\end{equation}
for each $j$.
Construct a sequence of cohomology classes $k=k_0\sim k_1 \sim \ldots \sim k_n$, where $k_{j+1}$ is obtained from $k_j$ by choosing any vertex $v$  weighted by $e$ such that $$\langle k_j,v \rangle = -e,$$ and letting $k_{j+1}=k_j+2PD(v)$. This algorithm ends at the \emph{terminal} cohomology class $k_n= \ell$ where either
\begin{itemize}
\item  we have
\begin{equation}
\label{goodpath}
e_j \leq \langle \ell ,v_j \rangle \leq -e_j+2
\end{equation}
at each $j$ or
\item there is some $m$ for which
\begin{equation}
\label{badpath}
\langle \ell ,v_m \rangle > -e_m.
\end{equation}
\end{itemize}

A sequence of cohomology classes $k=k_0\sim k_1 \sim \ldots \sim k_n$ is called a \emph{full path}. We say that the initial cohomology class $k$ supports a \emph{good full path}, (and respectively a \emph{bad full path}) if its terminal cohomology class satisfies \eqref{goodpath} (respectively \eqref{badpath}).

\begin{remark}
\label{badhereditary} Bad full paths are hereditary. If an initial cohomology class known to support a bad full path in any subgraph then it has a bad full path in the ambient graph as well. 
\end{remark}

\subsection{Ozsv\'ath-Szab\'o $d$-invariant formula}
The \emph{valency} of a vertex $v_j \in \mathrm{Vert}(G)$ is the number of of edges containing $v_j$. A vertex $v_j \in \mathrm{Vert}(G)$ is said to be a \emph{bad vertex} of the plumbing graph if its valency strictly greater than the minus of corresponding Euler number.

Let $\mathfrak{C}$ denote the set of characteristic cohomology classes satisfying \eqref{adjunction}. Let $\mathfrak{C'}$ also denote the set of terminal characteristic cohomology classes of good full paths. Then the formula of $d$-invariant is given as follows.

\begin{theorem} [\cite{OS03a}, Corollary 1.3; \cite{OS03b}, Proposition 4.2]
\label{correctionterm}
Let $G$ be a unimodular and negative-definite graph with one bad vertex. Let $Y(G)$ be the plumbed integral homology sphere corresponding to $G$. Then 
\begin{equation}
\label{corrterm}
\displaystyle d(Y(G))=  \underset{\{k \in\mathrm{Char}(G)\}}{\max}\frac{k^2+ \vert G \vert}{4} = \underset{ \{k \in \mathfrak{C}\}}{\max}\frac{k^2+ \vert G \vert}{4} = \underset{\{k \in \mathfrak{C'}\}}{\max}\frac{k^2+ \vert G \vert}{4}
\end{equation}
\end{theorem}

\begin{remark}
Note that $\mathrm{Char}(G)$ has infinitely many elements. On the other hand, $\mathfrak{C}$ has finitely many elements but it is considerably large set. Among them, $\mathfrak{C'}$ is the smallest one even though it is still not easy to describe all elements lying in $\mathfrak{C'}$.
\end{remark}

\begin{remark}
\label{Elk}
Due to the work of Elkies in \cite{Elk95}, we have $k^2 = \vert G \vert \mod 8$. Therefore, the Ozsv\'ath-Szab\'o $d$-invariant is an even integer.
\end{remark}

\section{Ozsv\'ath-Szab\'o $d$-invariant of Almost Simple Linear Graphs}
\label{simplelineargraphs}
\subsection{Simple Linear Graphs}
\label{simpleg}
Let us call a linear plumbing graph $A_t$ whose $t$ vertices are all weighted $-2$ a \emph{simple linear graph}. To analyse good full paths of such graphs, we apply the Ozsv\'ath-Szab\'o algorithm. Note that the boundary of a simple linear graph is not an integral homology sphere since a simple linear graph is never unimodular, but the Ozsv\'ath-Szab\'o algorithm is still valid. 

\begin{lemma}
\label{simple}
For a simple linear graph $A_t$, let $k$ be the initial characteristic cohomology class with $\langle k,v_i \rangle =2$ for more than one $i \in \{1, \ldots, t \}$. Then there is no good full path supported by $k$.  
\end{lemma}

\begin{proof}
Since $k$ is the initial characteristic cohomology class, $\langle k,v_i \rangle \in \{0,2\}$ for any $i$. For $t=2$, the initial cohomology class $k=(2,2)$ supports a bad full path for $A_2$ whose terminal cohomology class is obtained by adding $2PD(v_0)$ to $k$. Suppose $t>2$. If $k$ is the initial cohomology class of the form $(2,0,0,\ldots,0,0,2)$ then under the Ozsv\'ath-Szab\'o algorithm the bad full path starting from $K$ is shown as below.
$$k=(2,0,0,\ldots,0,0,2) \sim (-2,2,0,\ldots,0,0,2) \sim \ldots \sim (0,0,0,\ldots,-2,2,2).$$
For the general case, we consider the restricted part of initial cohomology class $k$ such that $\langle k,v_i \rangle= \langle k,v_j \rangle =2$ and $\langle k,v_s \rangle =0$ for $i<s<j$. Then again under the Ozsv\'ath-Szab\'o algorithm this cohomology class admits a bad full path for subgraph of $A_m$ as shown above. By Remark~\ref{badhereditary}, so $k$ does for $A_t$.
\end{proof}

\begin{lemma}
\label{simple2}
For a simple linear graph $A_t$, terminal cohomology classes of good full paths are either $(0,0,\ldots,0)$ or 
\[ (0, \ldots,0,\here{-2}{fromhere},0,\ldots \ldots ,0). \]
\begin{tikzpicture}[remember picture, overlay]
\node[font=\scriptsize, below right=12pt of fromhere] (tohere) {$(t-s+1)^{th}$-entry};
\draw[Stealth-] ([yshift=-4pt]fromhere.south) |- (tohere);
\end{tikzpicture}
\noindent%
\\
starting from the initial cohomology class
\noindent%
\[
(0, \ldots ,0,\here{2}{fromhere},0,  \ldots,0)
\]
\begin{tikzpicture}[remember picture, overlay]
\node[font=\scriptsize, below right=12pt of fromhere] (tohere) {$s^{th}$-entry};
\draw[Stealth-] ([yshift=-4pt]fromhere.south) |- (tohere);
\end{tikzpicture}
\end{lemma}

\begin{proof}
We start with some simple observations. The Ozsv\'ath-Szab\'o algorithm yields the following equivalences.
\begin{align*}
\mathrm{MV1} \ : \ & (2,0, \ldots, 0) \sim (0,\ldots0,-2), \ (0,0, \ldots, 2) \sim (-2,0, \ldots, 0), \\
\mathrm{MV2} \ : \ & (*, \ldots, *, -2,2,0, \ldots, 0) \sim (*, \ldots ,*, 0,0, \ldots, 0,-2), \\
\mathrm{MV3} \ : \ & (0, \ldots, 0, \here{2}{fromthere}, 0, \ldots,0, \here{-2}{fromhere}, *, \ldots,*)  \sim (0, \ldots, 0, \here{2}{fromtthere}, 0, \ldots,0, \here{-2}{fromttthere}, *, \ldots,*), \\
\begin{tikzpicture}[remember picture, overlay]
\node[font=\scriptsize, below left=8pt of fromthere] (tothere) {$s_0^{th}$};
\draw[Stealth-] ([yshift=-4pt]fromthere.south) |- (tothere);
\end{tikzpicture}
\begin{tikzpicture}[remember picture, overlay]
\node[font=\scriptsize, below right=8pt of fromhere] (tohere) {$t_0^{th}$};
\draw[Stealth-] ([yshift=-4pt]fromhere.south) |- (tohere);
\end{tikzpicture}
\begin{tikzpicture}[remember picture, overlay]
\node[font=\scriptsize, below left=8pt of fromtthere] (totthere) {$(s_0 - 1)^{th}$};
\draw[Stealth-] ([yshift=-4pt]fromtthere.south) |- (totthere);
\end{tikzpicture}
\begin{tikzpicture}[remember picture, overlay]
\node[font=\scriptsize, below right=8pt of fromttthere] (tottthere) {$(t_0 - 1)^{th}$};
\draw[Stealth-] ([yshift=-4pt]fromttthere.south) |- (tottthere);
\end{tikzpicture}
\\
\mathrm{MV4} \ : \ & (2,0, \ldots, 0,-2,*, \ldots, *) \sim (0, \ldots, 0, -2,0,* \ldots,*),
\end{align*} where $0, \ldots, 0$ represents a string of zeros and $*, \ldots, *$ represents an arbitrary string which does not change under the indicated moves.

Here, we give a step by step explanation of the first part in $\mathrm{MV1}$ and leave the verification of the rest of the moves to the reader as an exercise. At each step, we add the twice the Poincar\'e dual of the vertex at which the evaluation is $+2$. After this, each $+2$ turns into a $-2$ and evaluation at all the neighbouring vertices are increased by $2$. Hence we add all the vertices in order from left to right.
\begin{align*}
& (2,0,0 \ldots, 0) \sim (-2,2,0,\ldots0) \sim (0,-2,2, \ldots, 0) \\
& \sim \ldots \sim (0,0,0, \ldots, 0,-2,2) \sim (0,\ldots, 0,-2).
\end{align*}

Let $k$ be an initial characteristic cohomology class of $A_t$. Then $\langle k,v_i \rangle \in \{0,2\}$ for any $i$. By Lemma~\ref{simple}, $k$ is either of the form $(0,0,\ldots,0)$ or of the form $(0,\ldots,0,2,0,\ldots,0)$, where $2$ is the $s^{th}$-entry. For the former form of $k$, the terminal cohomology class is equal to $(0,0,\ldots,0)$. For the latter one, the proof follows from the move $\mathrm{MV1}$ if $s=0$ or $s=t$. Assume otherwise. Then we obtain
\begin{align*}
(0, \ldots \ldots,0,\here{2}{fromone},0,\ldots \ldots ,0) & \sim (0, \ldots, 0,\here{2}{fromtwo},\here{-2}{fromthree},2,0,\ldots,0) \ \ \ \ \ \ \ \ \ [\mathrm{Add} \ 2\mathrm{PD}(v_s)] \\
\begin{tikzpicture}[remember picture, overlay]
\node[font=\scriptsize, below right=8pt of fromone] (tohere) {$s^{th}$};
\draw[Stealth-] ([yshift=-4pt]fromone.south) |- (tohere);
\end{tikzpicture}
\begin{tikzpicture}[remember picture, overlay]
\node[font=\scriptsize, below left=8pt of fromtwo] (tohere) {$(s-1)^{th}$};
\draw[Stealth-] ([yshift=-4pt]fromtwo.south) |- (tohere);
\end{tikzpicture}
\begin{tikzpicture}[remember picture, overlay]
\node[font=\scriptsize, below right=8pt of fromthree] (tothree) {$s^{th}$};
\draw[Stealth-] ([yshift=-4pt]fromthree.south) |- (tothree);
\end{tikzpicture}
\\
& \sim (0, \ldots,0,\here{2}{fromfour},0,\ldots,0,\here{-2}{fromfive}) \ \ \ \ \ \ \ \ \ \ \ \ [\mathrm{MV2}] \\
\begin{tikzpicture}[remember picture, overlay]
\node[font=\scriptsize, below left=8pt of fromfour] (tohere) {$(s-1)^{th}$};
\draw[Stealth-] ([yshift=-4pt]fromfour.south) |- (tohere);
\end{tikzpicture}
\begin{tikzpicture}[remember picture, overlay]
\node[font=\scriptsize, below right=8pt of fromfive] (tohere) {$t^{th}$};
\draw[Stealth-] ([yshift=-4pt]fromfive.south) |- (tohere);
\end{tikzpicture}
\\
& \sim (2,0,\ldots,0,\here{-2}{fromsix},0,\ldots,0) \ \ \ \ \ \ \ \ \ \ \ \ [(s-2)-\mathrm{times} \ \mathrm{MV3}] \\
\begin{tikzpicture}[remember picture, overlay]
\node[font=\scriptsize, below right=8pt of fromsix] (tohere) {$(t-s+2)^{th}$};
\draw[Stealth-] ([yshift=-4pt]fromsix.south) |- (tohere);
\end{tikzpicture}
\\
& \sim (0,\ldots,0,\here{-2}{fromseven},0,\ldots,0) \ \ \ \ \ \ \ \ \ \ \ \ \ \ \ [\mathrm{MV4}] \\
\begin{tikzpicture}[remember picture, overlay]
\node[font=\scriptsize, below right=8pt of fromseven] (tohere) {$(t-s+1)^{th}$};
\draw[Stealth-] ([yshift=-4pt]fromseven.south) |- (tohere);
\end{tikzpicture} 
\end{align*}
\end{proof}

\subsection{Almost Simple Linear Graphs}
 
\begin{proposition}
\label{Plumb}
Let $p,q$ and $r$ be pairwise relatively prime, ordered, positive integers satisfying \eqref{pqr}. Then Brieskorn sphere $Y=\Sigma(p,q,r)$ is the boundary of the plumbing graph shown in Figure~ \ref{fig:plumb}.
\end{proposition}

Notice that removing a single vertex from the plumbing graph in Figure~\ref{fig:plumb} leads to a simple linear graph. We call such graphs \emph{almost simple linear}.

\begin{proof}[Proof of Proposition~\ref{Plumb}]
To find this plumbing graph, we follow Section~\ref{Plumbings}. Given $p,q$ and $r$, we first investigate unique integers $e_0,p',q',r'$ solving the Diophantine equation
\begin{equation}
\label{0}
e_0pqr+p'qr+pq'r+pqr'=-1
\end{equation} 
where $1\leq p' \leq p-1$, $1\leq q' \leq q-1$ and $1\leq r' \leq r-1$.

Taking $\mathrm{mod}\ p$, $\mathrm{mod} \ q$ and $\mathrm{mod} \ r$ reductions of both sides Equation \eqref{pqr} and \eqref{0} respectively, one can immediately see that $e_0=-2$, $p'=1$, $q'=q-1$ and $r'=r-1$. 

Hence we find continued fraction expansions $\frac{p}{p'}$, $\frac{q}{q'}$ and $\frac{r}{r'}$ describing weights of branches of plumbing graph as follows.
\begin{align*}
\frac{p}{1}&=[p], \\
\frac{q}{q-1}&=[\underbrace{2:2:\dots:2}_{q-1}],\text{ and}\\
\frac{r}{r-1}&=[\underbrace{2:2:\dots:2}_{r-1}].
\end{align*}
\end{proof}

\begin{remark}
\label{triplet}
There are many different triplets $(p,q,r)$ satisfying the identity \eqref{pqr}.  Indeed, if we choose $s$ to be any positive divisor of $p^2-1$, then the triplet $(p,p+s,p+(p^2-1)/s)$ solves \eqref{pqr}. For example, all of the following triplet satisfy \eqref{pqr}
\begin{itemize}
\item $(p,q,r)=(2n+1,4n+1,4n+3)$ for $n \geq 1$,
\item $(p,q,r)=(2n+1,3n+2,6n+1)$ for $n \geq 1$,
\item $(p,q,r)=(2n+1,3n+1,6n+5)$ for $n \geq 1$,
\item $(p,q,r)=(4n+3,5n+4,20n+11)$ for $n \geq 1$,
\item $(p,q,r)=(p,p+1,p^2 + p-1)$ for $p \geq 2$,
\item $(p,q,r)=(F_{2n+1},F_{2n+2},F_{2n+3})$, where $F_*$ is $*^{th}$ Fibonacci number.
\end{itemize}
\end{remark}

\begin{figure}[ht]  \begin{center}
		\includegraphics[width=0.5\textwidth]{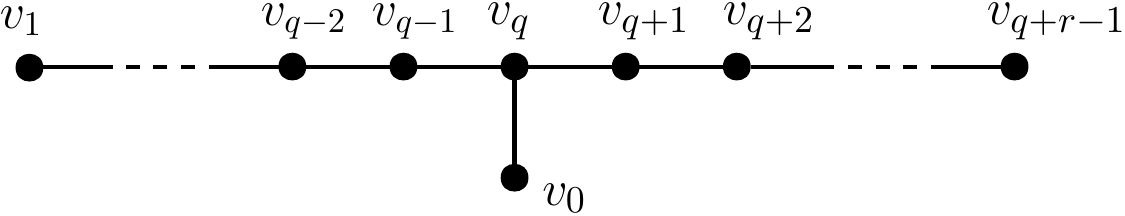}       
		\caption{The labelling of vertices}
\label{fig:decoration}
	\end{center}
\end{figure}

\subsection{Computation of Ozsv\'ath-Szab\'o $d$-invariant}
Throughout this subsection, $G$ denotes the almost simple linear plumbing graph in Figure~\ref{fig:plumb} and $Y=\Sigma(p,q,r)$ is the Brieskorn sphere described in Proposition~\ref{Plumb}. We label the vertices of plumbing graph $G$ as shown in Figure~\ref{fig:decoration}. Note that $|G|=q+r$.

\begin{proposition}
\label{deven}
For an even $p$, we have that $d(Y)=(q+r)/4.$ 
\end{proposition}

\begin{proof}
When $p$ is even, $X$ has even intersection form and thus $k=0$ is a characteristic cohomology class. Since the intersection form is negative definite, $k=0$ is clearly maximizes the expression \eqref{corrterm}. Hence, $d(Y)=(q+r)/4$ by Theorem~\ref{correctionterm}.  
\end{proof}

From now on, we assume that $p$ is odd. For $a \in \mathcal{A} = \{ -p, -p+2, \ldots, p-2, p \}$ and $m \in  \{0,1, \ldots ,(p-1)/2 \}$, define $k_{a,m}$ to be the characteristic cohomology class indicated in Figure~\ref{fig:charvector}, where the decoration of each vertex $v_i$ indicates the number $\langle k_{a,m},v_i \rangle$. Note that $k_{a,0}$ does not include a $-2$.  

\begin{lemma}
\label{max}
For $Y$, the maximum of the formula \eqref{corrterm} is achieved among characteristic cohomology classes $k_{a,m}$. 
\end{lemma}

\begin{figure}[ht]  \begin{center}
		\includegraphics[width=0.6\textwidth]{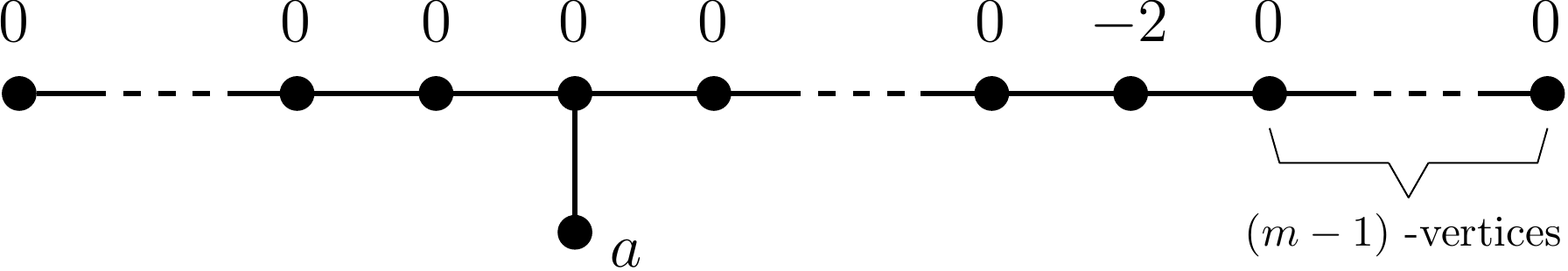}       
		\caption{Characteristic cohomology classes $k_{a,m}.$}
\label{fig:charvector}
	\end{center}
\end{figure}

\begin{proof}
Let $k$ be a initial characteristic cohomology class of the good full path for the plumbing graph $G$. Then we have $$\langle k,v_0 \rangle \in \mathcal{A} \ \text{and} \ \langle k,v_i \rangle \in \{0,2\}$$ for each $i \neq 0$.

Due to Lemma~\ref{simple2} if $k$ admits a good full path, then it is either of the form $-k_{a,m}$ where $a \in \mathcal{A}$ and $m \in  \{0,1, \ldots ,r-2, r-1 \}$; or of the form $\tilde{k}_{a,m}$ where $a \in \mathcal{A}$ and $m \in  \{0,1, \ldots , q-1 \}$ respectively; see Figure~\ref{fig:charvector} and Figure~\ref{fig:charvector2}.  

\begin{figure}[ht]  \begin{center}
		\includegraphics[width=0.6\textwidth]{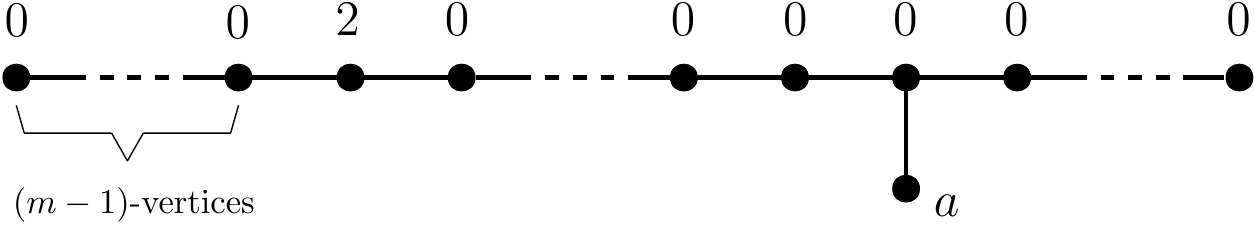}       
		\caption{Characteristic cohomology classes $\tilde{k}_{a,m}.$}
\label{fig:charvector2}
	\end{center}
\end{figure}

If $k = -k_{a,m}$, then $\mathrm{deg}(k) = \mathrm{deg}(-k_{a,m}) = \mathrm{deg}(k_{a,m})$. If $k=\tilde{k}_{a,m}$, then apply the moves at the end of the proof of Lemma~\ref{simple2} to see $k=\tilde{k}_{a,m} \sim k_{-a+2m, m}$. If $a+2m \geq p$, then the full path starting from $k$ is bad. Otherwise, it is a good one and we have $\mathrm{deg}(k) = \mathrm{deg}(\tilde{k}_{a,m}) = \mathrm{deg}(k_{-a+2m,m})$. Therefore, the maximum of~\eqref{corrterm} is achieved among characteristic cohomology classes $k_{a,m}$.
\end{proof}

\begin{lemma} 
\label{max2} 
We have the formula
\begin{equation}
\label{fan}
k^2_{a,m}= -(q+r)a^2 +4aqm -4(q-p)m^2 -4m.
\end{equation}

\end{lemma}

\begin{proof}
We will compute $k^2_{a,m}$ as in the end of Section~\ref{Plumbings}. Thanks to the fact that most of the decoration in $k_{a,m}$ are zeros, we do not need to describe all the entries of $I^{-1}$. 

Note that $\langle k_{a,m}, v_0 \rangle = a$ and $\langle k_{a,m}, v_{q+r-m-1} \rangle = -2$. Then
$$k^2_{a,m}= a^2I^{-1}_{v_0v_0} -4a I^{-1}_{v_0v_{q+r-m-1}} +4I^{-1}_{v_{q+r-m-1}v_{q+r-m-1}}.$$

Following Section~\ref{Plumbings}, we will find entries of $I^{-1}$. Let $p_{v_0v_0}$, $p_{v_0v_{q+r-m-1}}$ and $p_{v_{q+r-m-1}v_{q+r-m-1}}$ be unique minimal paths in $G$ connecting $v_0$ to itself, $v_0$ to $v_{q+r-m-1}$, and $v_{q+r-m-1}$ to itself respectively. Deleting $p_{v_0v_0}$ in $G$, we get that $I^{-1}_{v_0v_0}= -(q+r)$. Also deleting $p_{v_0v_{q+r-m-1}}$ in $G$, we similarly obtain $I^{-1}_{v_0v_{q+r-m-1}}=-qm$. Finally, $I^{-1}_{v_{q+r-m-1}v_{q+r-m-1}}$ is the minus of product of $-m$ and the one which is equal to the numerator of the following quantity.

\begin{align*}
-2+\frac{1}{p}+\frac{q-1}{q}+\frac{r-m-1}{r-m}= - \frac{m(q-p)+1}{pq(r-m)}. 
\end{align*}
Here, we used the fact that $p,q$ and $r$ satisfy Equation \eqref{pqr}. Thus, $I^{-1}_{v_{q+r-m-1}v_{q+r-m-1}}=-(q-p)m^2-m$. Hence, Equation \eqref{fan} holds. 
\end{proof}

\begin{proof}[Proof of Theorem \ref{theo:latt}]
Apply Ozsv\'ath-Szab\'o $d$-invariant formula in Theorem \ref{correctionterm} to $\Sigma(p,q,r)$. We showed in Lemma \ref{max} that the maximum degree of the characteristic classes  is achieved among $k_{a,m}$ where $(a,m)$ runs through the lattice points. In Lemma \ref{max2} we proved that $k_{a,m}^2=f(a,m)$ for all $(a,m)\in \mathfrak{L}$, hence the first equation in the statement is true. The second equation trivially follows from the first one.
\end{proof}

\subsection{A closer look at $\mathfrak{L}\cap R$}
It is now natural to ask which lattice points, besides $(1,1)$, are contained in the region $R$. To address this question we look for some simpler descriptions of $R$. Define 
\begin{align}
\nonumber g(x,y)&=f(1,1)-f(x,y)\\
&=(q+r)x^2-4qxy+4(q-p)y^2+4y-(q+r)+4(p-1) .
\end{align} 
\noindent Recall that $R=\{(x,y)\,:\, -p\leq x\leq p,\; 0\leq y \leq (p-1)/2,\; g(x,y)\leq 0 \}$.  Fix $y\in[0,(p-1)/2]$, then $g(x,y)$ is quadratic in $x$ with positive leading term. Therefore $g(x,y)\leq 0$ if and only if $\Delta(y)\geq 0$ and 
$$x \in \mathcal{I}_y = \left [ \frac{4qy-\sqrt{\Delta (y)}}{2(q+r)}, \frac{4qy+\sqrt{\Delta (y)}}{2(q+r)}\right ] $$  
\noindent where
\begin{align}
\nonumber \Delta(y)&=(-4qy)^2-4(q+r)(4(q-p)y^2+4y-(q+r+4(p-1)))\\
&=4[2y-(q+r)]^2-16(q+r)(p-1).
\end{align}
Notice that $\Delta'(y)<0$, so $\Delta(y)$ is a decreasing function of $y$. We would like to decide whether the interval $\mathcal{I}_m$ contains an odd integer for $m\in \{0,1,\dots,(p-1)/2\}$. Define
\begin{align}
\mathfrak{c}(m) = \frac{2qm}{q+r}\text{ and }
\mathfrak{r}(m) = \frac{\sqrt{\Delta(m)}}{2(q+r)} = \sqrt{\left ( \frac{2m}{q+r} -1 \right )^2 - \frac{4(p-1)}{q+r}}
\end{align}
so that the interval $\mathcal{I}_m$ has center $\mathfrak{c}(m)$ and radius $\mathfrak{r}(m)$. Hence we see that a lattice point $(a,m)$ is contained in $R$ if and only if $\Delta(m)\geq 0$ and  $|a-\mathfrak{c}(m)|\leq \mathfrak{r}(m)$. Notice that $\mathfrak{r}(0)<1$, so $\mathfrak{r}(m)<1$ for every $m$. Therefore, $R$ contains no lattice point in the slices $m=0,1$ except the point $(1,1)$. Moreover for every fixed $y$, the function $f(x,y)$ is quadratic in $x$ with negative leading term, and the critical point is $x=2qy/(q+r)$.  Hence on each horizontal slice $y=m$, the maximum of $f(a,m)$ is at those odd integers $a$ which are closest to $\mathfrak{c}(m)$.  Let $\mathfrak{d}(m)$ denote the distance from $\mathfrak{c}(m)$ to the closest odd integer $a$ so that $\mathfrak{d}(m)\leq \mathfrak{r}(m)$. We summarize all these observations in the following proposition.
\begin{proposition}
\label{lattice}
A lattice point $(a,m)$ is contained in $R$ if and only if $(a,m)=(1,1)$ or all of the following conditions are satisfied
\begin{enumerate}
\item $m\geq 2$,
\item $\Delta(m)\geq 0$,
\item $\mathfrak{d}(m)\leq \mathfrak{r}(m)$,
\item $a$ is the closest odd integer to $\mathfrak{c}(m)$. 
\end{enumerate} 
\end{proposition}	

\section{Proof of Theorem~\ref{theo:d-inv}}
\label{calculation}
\subsection*{Case 1} Take $p=2n+1$, $q=4n+1$, $r=4n+3$ and observe that $\Delta(m) = 16(m^2 - 8mn -4m +8n +4)$ which is negative when $m$ is between $4n+2 \mp \sqrt{(1+ \frac{1}{2n})}$. The smaller root satisfies $4n+2 - \sqrt{(1+ \frac{1}{2n})} < 2$ so $\Delta(m) <0$ for all $m\geq 2$. By Proposition~\ref{lattice} and Theorem~\ref{theo:latt}, we have $$d(\Sigma(2n+1,4n+1,4n+3)) = \frac{f(1,1) + q+r}{4} = \frac{-4 + 8n+4}{4} = 2n.$$

\begin{remark}
Case 1 can be also proved by using Elkies' theorem stated in Remark~\ref{Elk}; namely, $k^2 = 4 \mod 8$. By negative-definiteness the largest possible value for $k^2$ is $-4$ which is realized by $k_{1,1}$.
\end{remark}

\subsection*{Case 2} Take $p=2n+1$, $q=3n+2$, $r=6n+1$. For $m \in \{ 0,1,\ldots, n \}$, it is easy to verify $$\mathfrak{c}(m) = \left (1- \frac{1}{3n+1} \right ) \frac{2}{3}m$$ so that
$$\mathfrak{d}(m) = \begin{cases} 
      1- \frac{2m}{9n+3}, & \text{if} \ m = 0 \mod 3 , \\
      \frac{1}{3} - \frac{2m}{9n+3}, & \text{if} \ m = 1 \mod 3, \\
      \frac{1}{3} + \frac{2m}{9n+3}, & \text{if} \ m = 2 \mod 3, 
   \end{cases}$$ and $$\mathfrak{r}(m)^2 = \left (\frac{2m}{9n+3} -1 \right )^2 - \frac{8n}{9n+3}.$$
   
Set $F(m) = \left ( \frac{1}{3} - \frac{2m}{9n+3} \right )^2 - \mathfrak{r}(m)^2$ so that $\mathfrak{d}(m)^2 - \mathfrak{r}(m)^2 \geq F(m) = \frac{8m-8}{3(9n+3)}.$ The function $F$ is linear in $m$ and has root at $m=1$ so it satisfies $F(m)\geq 0$ for all $m\geq 1$, implying in particular that $\mathfrak{d}(m) >\mathfrak{r}(m)$ holds for every $m \geq 2$.  

Note that $f(1,1)= -n-3$. By Proposition~\ref{lattice} and Theorem~\ref{theo:latt}, we have $$d(\Sigma(2n+1,3n+2,6n+1)) = \frac{f(1,1) + q+r}{4} = \frac{-n-3 + 9n+3}{4} = 2n.$$

\subsection*{Case 3} Take $p=2n+1$, $q=3n+1$, $r=6n+5$. Then for $m \in \{ 0,1,\ldots, n \}$ the quantities $\mathfrak{c}(m)$, $\mathfrak{d}(m)$ and $\mathfrak{r}(m)$ are easily found as follows. $$\mathfrak{c}(m) = \left (1- \frac{1}{3n+2} \right ) \frac{2}{3}m$$ so that $$\mathfrak{d}(m) = \begin{cases} 
      1 + \frac{2m}{9n+6}, & \text{if} \ m = 0 \mod 3 , \\
      \frac{1}{3} + \frac{2m}{9n+6}, & \text{if} \ m = 1 \mod 3, \\
      \frac{1}{3} - \frac{2m}{9n+6}, & \text{if} \ m = 2 \mod 3, 
   \end{cases}$$ and $$\mathfrak{r}(m)^2 = \left ( \frac{2m}{9n+6} -1 \right )^2 - \frac{8n}{9n+6}.$$

Set $G(m) = \left ( \frac{1}{3} - \frac{2m}{9n+6} \right )^2 - \mathfrak{r}(m)^2$ so that $\mathfrak{d}(m)^2 - \mathfrak{r}(m)^2 \geq G(m) = \frac{8m-16}{3(9n+6)}.$ The function $G$ is linear in $m$ and has root at $m=2$ so it satisfies $G(m)\geq 0$ for all $m\geq 2$, implying in particular that $\mathfrak{d}(m) >\mathfrak{r}(m)$ holds for every $m \geq 3$. For $m=2$, the closest odd integer to $\mathfrak{c}(2)$ is $1$. Therefore, $(1,2)$ is the only other candidate lattice point besides $(1,1)$.

Since $f(1,1)=-n-6=f(1,2)$, by Proposition~\ref{lattice} and Theorem~\ref{theo:latt}, we have $$d(\Sigma(2n+1,3n+1,6n+5)) = \frac{f(1,1) + q+r}{4} = \frac{-n-6 + 9n+6}{4} = 2n.$$

\subsection*{Case 4} Take $p=4n+3$, $q=5n+4$, $r=20n+11$. Then for $m \in \{ 0,1,\ldots, 2n+1 \}$ the quantities $\mathfrak{c}(m)$, $\mathfrak{d}(m)$ and $\mathfrak{r}(m)$ are simply computed as follows. $$\mathfrak{c}(m) = \left (1+ \frac{1}{5n+3} \right ) \frac{2}{5}m$$ so that 
$$\mathfrak{d}(m) = \begin{cases} 
      1 - \frac{2m}{25n+15}, & \text{if} \ m = 0 \mod 5 , \\
      \frac{3}{5} - \frac{2m}{25n+15}, & \text{if} \ m = 1 \mod 5, \\
      \frac{1}{5} - \frac{2m}{25n+15}, & \text{if} \ m = 2 \mod 5, \\
     \frac{1}{5} + \frac{2m}{25n+15}, & \text{if} \ m = 3 \mod 5, \\
      \frac{3}{5} + \frac{2m}{25n+15}, & \text{if} \ m = 4 \mod 5, \\ 
   \end{cases}$$ and $$\mathfrak{r}(m)^2 = \left ( \frac{2m}{25n+15} -1 \right )^2 - \frac{16n+8}{25n+15}.$$
   
We split the argument into several subcases and analyse them by using Proposition~\ref{lattice} and Theorem~\ref{theo:latt}. First assume that $m \neq 2 \ \text{or} \ 3 \ \text{mod} \ 5$. Define $H(m) = \left ( \frac{3}{5} - \frac{2m}{25n+15} \right) - \mathfrak{r}(m)^2$ so that $\mathfrak{d}(m)^2 - \mathfrak{r}(m)^2 \geq H(m) = \frac{8m-8}{25(5n+3)}$. The function $H$ is linear in $m$ and has a root at $m=1$. Since $H(m) >0$, we have $\mathfrak{d}(m) > \mathfrak{r}(m)$ for all $m>1$. For $m=1$, the closest odd integer to $\mathfrak{c}(1)$ is $1$. Hence $(1,1)$ is the only candidate lattice point in this subcase.  

Next suppose that $m= 2 \ \text{mod} \ 5$. Write $m=5k+2$ where $k \in \{ 0,1, \ldots, \lfloor \frac{2n-1}{5} \rfloor \}$. The closest odd integer to $\mathfrak{c}(5k+2)$ is $2k+1$ . Then the maximum value of the function $f(2k+1, 5k+2) = -16k -n-7$ is achieved when $k=0$. Thus $(1,2)$ is the only candidate lattice point in this subcase. 

Finally assume that $m= 3 \ \text{mod} \ 5$. Write $m=5k+3$ where $k \in \{ 0,1, \ldots, \lfloor \frac{2n-2}{5} \rfloor \}$. The closest odd integer to $\mathfrak{c}(5k+3)$ is $2k+1$ . Then the maximum value of the function $f(2k+1, 5k+3) = -24k -n-15$ is achieved when $k=0$. Thus $(1,3)$ is the only candidate lattice point in this subcase. 
   
Since $f(1,2)=-n-7 > f(1,3) = -n-15 > f(1,1)= -9n-7$, we have $$d(\Sigma(4n+3,5n+4,20n+11)) = \frac{f(1,2) + q+r}{4} = \frac{-n-7 + 25n+15}{4} = 6n+2.$$

\subsection*{Case 5} Take $p=2n+1$, $q=2n+2$, $r=4n^2+6n+1$ for $n \geq 1$. For $m \in \{0,1, \ldots, n \}$, the closest odd integer to $\mathfrak{c}(m) = \frac{4(n+1)}{4n^2 + 8n +3}$ is $1$. We look at $$f(1,m)= -4(m-n)^2 + 4(m-2n) -3.$$ The maximum is achieved at $m=n$. Since $f(1,n) = -4n-3$, by Proposition~\ref{lattice} and Theorem~\ref{theo:latt}, we have $$d(\Sigma(2n+1,2n+2,4n^2+6n+1)) = \frac{f(1,n) + q+r}{4} = \frac{-4n-3 + 4n^2 + 8n +3}{4} = n^2 +n.$$
\begin{flushright}
$\square$  
\end{flushright}
\begin{remark}
When $p=2n$, $q=2n+1$, and $r=4n^2-1$, then we have $d(2n,2n+1,4n^2+2n-1) = n^2 +2n$ by Proposition~\ref{deven}. Together with Case 5, this is a special part of the previous result of Borodzik and N\'emethi \cite[Theorem 1.6]{BN13}, which describes Heegaard Floer homology of $(+1)$-surgery of torus knots $T(p,p+1)$ for $p \geq 2$.
\end{remark}

\begin{remark}
Using our techniques one can simultaneously compute $d$-invariants of any infinite family of Brieskorn spheres $\Sigma(p,q,r)$ satisfying \eqref{pqr} provided that one knows the odd integer closest to $\mathfrak{c}(m)$.
\end{remark}

\section{Acknowledgements}
This research was  supported by BAGEP award of the Science Academy and Bo\u{g}azi\c{c}i University Research Fund Grant Number 12482.

\bibliography{references}
\bibliographystyle{amsalpha}

\end{document}